\documentclass[reqno]{amsart}
\usepackage{amssymb,amsmath,amsfonts,amsthm} 
\usepackage{enumitem}    
\usepackage{verbatim}
\usepackage{hyperref}    
\usepackage{graphicx}
\usepackage{tikz}
\usepackage[capitalize]{cleveref} 
\usepackage{array,multirow}

\usepackage[ruled, lined, linesnumbered, commentsnumbered, longend]{algorithm2e}

\usepackage[numbers,sort&compress]{natbib}
\hypersetup{colorlinks=true,linkcolor=green,pdfborderstyle={/S/U/W 1}}

\newcommand{\ot}{\leftarrow}
\newcommand{\ceil}[1]{{\lceil #1 \rceil}}

\theoremstyle{plain}
\newtheorem{theorem}  {Theorem}  [section]
\newtheorem{lemma}  [theorem]   {Lemma}
\newtheorem{corollary}[theorem] {Corollary}

\newtheorem{fact} [theorem] {Fact}
\newtheorem{proposition} [theorem] {Proposition}
\newtheorem{claim} [theorem] {Claim}

\theoremstyle{definition}
\newtheorem{remark}[theorem] {Remark}
\newtheorem{definition}[theorem] {Definition}
\newtheorem{example}[theorem] {Example}

\newcommand{\claimproof}{\renewcommand{\qedsymbol}{$\diamond$}}

\usetikzlibrary{decorations.markings}
\tikzset{
  vert/.style={circle, draw=black!100,fill=black!100,thick, inner sep=0pt, minimum size=2mm}, 
  empty/.style={draw=none, fill=none, minimum size=0mm, inner sep=0pt},
  labvert/.style={circle, draw=black!100,fill=none,thick, inner sep=2pt, minimum size=2mm},
  C/.style={circle, draw=black, fill=black, inner sep = 1pt},
  D/.style={circle, draw=black, fill=black!10, minimum size=5mm},
  arc/.style = {->,> = latex'},
  f/.style={very thick,decoration={markings,mark=at position 0.5 with {\arrow{>}}},postaction={decorate}},
  b/.style={very thick,decoration={markings,mark=at position 0.5 with {\arrow{<}}},postaction={decorate}}
}

 \newcommand{\Loop}[2][0]{\draw (#2) edge[min distance=5mm,in=#1*90+30,out=#1*90-30,looseness=4] (#2);}
 \newcommand{\Vlabel}[4][.3cm]{\draw node[empty, #3 of = #2, node distance = #1] () {$#4$};}

\DeclareMathOperator{\Hom}{Hom}
\DeclareMathOperator{\Mon}{Mon}
\DeclareMathOperator{\Neq}{Neq}

\DeclareMathOperator{\Recon}{Recon}

\DeclareMathOperator{\Aux}{A}

\renewcommand{\hat}{\widehat}

\begin{document}

\thispagestyle{empty}  
\title{Reflexive Digraph Reconfiguration by Orientation Strings}

\author{David Emmanuel Pazmi\~{n}o Pullas}
\address{Université du Québec à Montréal}
\email{pazmino\_pullas.david\_emmanuel@courrier.uqam.ca}
\thanks{}

\author{Mark Siggers}
\address{Kyungpook National University Mathematics Department,
80 Dae-hak-ro, Daegu Buk-gu, South Korea, 41566}
\email{mhsiggers@knu.ac.kr}
\thanks{The second author is supported by Korean NRF Basic Science Research Program (2022-R1D1A1A09083741) funded by the Korean government (MEST) and the Kyungpook National University Research Fund.}


\begin{abstract}
  The reconfiguration problem for homomorphisms of digraphs to a reflexive digraph cycle, which amounts to deciding if a `reconfiguration graph' is connected, 
  is known to by polynomially time solvable via a greedy algorithm based on certain topological requirements.   Even in the case that the instance digraph is a cycle of length $m$, the algorithm, being greedy, takes time $\Omega(m^2)$.  
  Encoding homomorphisms between two cycles as a relation on strings that represent the orientations of the cycles, we give a characterization  of the components of the reconfiguration graph in terms of these strings.
  The component under this characterization can be computed in linear time and logarithmic space.  In particular, this solves the reconfiguration problem for homomorphisms of cycles to cycles in log-space.       
\end{abstract}

\maketitle

\section{Introduction}

A digraph $G$ is a binary relation `$\to$` on a set $V(G)$ of vertices. 
An ordered pair $uv$ is an {\em arc} of $G$ if $u \to v$, and is an {\em edge}, denoted $u \sim v$ if $uv$ or $vu$ is an arc. The arc $uu$ is a {\em loop.} 

The {\em underlying graph} of a digraph $G$ is the graph we get by replacing non-loop arcs with edges. A digraph is a {\em cycle,} a {\em path} or a {\em tree} if its underlying graph is. The length or girth of a digraph is that of its underlying graph, in particular cycles have length at least $3$.  
 
The $\Hom$-graph $\Hom(G,H)$ for two digraphs $G$ and $H$ is the digraph whose vertex set is the set of homomorphisms of $G$ to $H$, and in which $\phi \to \phi'$, for two homomorphisms $\phi$ and $\phi'$, if for all pairs $u,v$ of vertices of $V(G)$, $u \to v$ implies $\phi(u) \to \phi'(v)$. The {\em homomorphism reconfiguration problem $\Recon(H)$} for a digraph $H$ asks, for an instance $(G,\phi,\psi)$ consisting of a digraph $G$ and two homomorphisms $\phi,\psi \in \Hom(G,H)$, if there is a path between  $\phi$ and $\psi$ in $\Hom(G,H)$.  Such a path is called a 
{\em reconfiguration} of $\phi$ to $\psi$.


A digraph is {\em reflexive} if every vertex $v$ has a loop.  In this paper we consider the problem $\Hom(C,D)$ where both the target $C$, and the instance $D$ are reflexive digraph cycles.   The use of reflexive graphs $D$ as the target plays an obvious role. Adjacent vertices in the instance $C$ can map to the same vertex in $D$ under a homomorphism.  The use of reflexive instances $C$ plays a slightly less obvious role.   Two homomorphisms $\phi, \psi \in \Hom(G,H)$ that differ only on a single vertex $v$ are always adjacent if $v$ does not have a loop, but are only adjacent for reflexive $v$ if $\phi(v) \sim \psi(v)$.  This is also less important-- \cref{rmk:looplessC} explains how most of our results hold for instances $C$ that are not necessarily reflexive.

In \cite{BLS21} it was shown that for any reflexive cycle $D$, $\Recon(D)$ is polynomial time solvable for reflexive instances $G$.  Indeed, it was shown that there is a path between two homomorphism $\phi$ and $\psi$ in $\Hom(G,D)$ if and only if for every cycle $C \leq G$ there was a path between the restrictions of $\phi$ and $\psi$ to $C$ in $\Hom(C,D)$, and that this could be determined (for all cycles $C$ at the same time) in polynomial time.
 
 From this, we have as a special case that one can determine in time polynomial in $|V(C)|$, for digraph cycles $C$ and $D$, if two maps in $\Hom(C,D)$ are in the same component. The algorithm in \cite{BLS21} was not optimized, especially not for cycle instances $C$,  but for an instance $C$ of length $m$,  it would take time $O(m^2)$. While a polynomial algorithm is usually our goal in such problems, this was a little unsatisfying: one should have a more explicit description of the cycles $C$ for which $\Hom(C,D)$ is disconnected, and more generally for such cycles, should have a more explicit description of $\Hom(C,D)$.  To clarify what we mean by `a more explicit description', we now give the restriction of our main result to reflexive symmetric cycles.

 The {\em wind} $w(\phi)$ (see Definition \ref{def:wind})  of a homomorphism $\phi: C \to D$ of reflexive symmetric cycles  counts the number of times $\phi$ winds the cycle $C$ around $D$.  This is an integer between $-m/n$ and $m/n$ where 
$C$ has length $m$ and $D$ has length $n$.  
 When $D$ has length $n \geq 4$, it is not hard to show that the wind is constant on components of $\Hom(C,D)$, and so $\Hom(C,D)$ is the disjoint union of the subgraphs $\Hom_w(C,D)$ induced by maps of wind $w$, as $w$ runs from $-m/n$ to $m/n$.  The following is quite simple to prove.

\begin{fact}\label{fact:mainsym}
  Let $C$ and $D$ be reflexive symmetric cycles of lengths $m$ and $n$ respectively, with $4 \leq n \leq m$.  The graph $\Hom_w(C,D)$  consists of  
  \begin{enumerate}
    \item a single component if $0 \leq |w| < m/n$, 
    \item $n$ isolated vertices if $0 < |w| = m/n$,
   \end{enumerate}
  and is empty if $m/n < |w|$.   
\end{fact}


When $C$ and $D$ are digraphs the situation is not so simple. The lengths of $C$ and $D$ are not enough to determine everything. Indeed, even in the case that
$0 < w = 1 < m/n$, the digraph $\Hom_1(C,D)$ can be empty, connected, or consist of many not necessarily trivial components.
Our main result is a digraph version of the above fact.  To state it, we need further definitions. Before we give them, we introduce them informally with an example.

 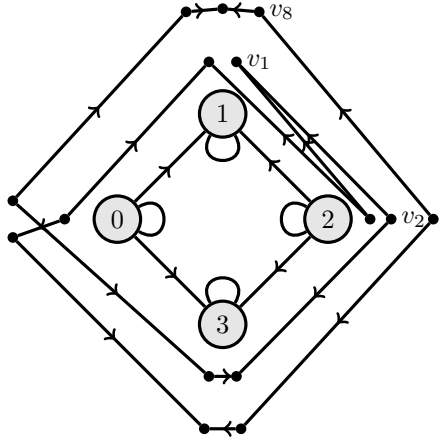
\begin{figure}
 \begin{center}
 \begin{tikzpicture}[very thick,scale=.7]

        \foreach \i in {0,1,2,3}{\draw (180-90*\i:2) node[D] (d\i){$\i$};}
        \foreach \orient[count=\i from 0][evaluate=\i as  \j using {int(int(round(mod(\i+1, 4))))}] in {f,b,f,b}{\draw[\orient] (d\i) -- (d\j);}
        \foreach \v/\p in {d0/0, d1/3, d2/2,d3/1}{\Loop[\p]{\v}};
       
        \foreach \i/\degree/\radius in {0/180/3, 1/95/3, 2/0/2.8, 3/85/3, 4/0/3.2, 5/275/3, 6/265/3, 7/175/4, 8/100/4, 9/90/4, 10/80/4, 11/0/4, 12/275/4, 13/265/4, 14/185/4}{\draw (\degree:\radius) node[C] (c\i){};}
        \foreach \orient[count=\i from 0][evaluate=\i as  \j using {int(int(round(mod(\i+1, 15))))}] in {f,b,f,b,f,b,b,f,f,b,b,f,f,b}{\draw[\orient] (c\i) -- (c\j);}
        \draw[f](c0) --  (c14);

        \Vlabel{c3}{right}{v_1} 
        \Vlabel{c4}{right}{v_2}
        \Vlabel{c10}{right}{v_8}

\end{tikzpicture}
 \caption{Example of a homomorphism of a digraph $15$-cycle $C$ to a reflexive digraph $4$-cycle $D$}\label{fig:1}
\end{center}
\end{figure}

\begin{example}
Consider the wind $2$ homomorphism of a digraph $15$-cycle $C$ to a reflexive digraph $4$-cycle $D$ on the vertex set  $\{0,1,2,3\}$ shown in Figure \ref{fig:1}.
As $D$ is reflexive, consecutive vertices of $C$ can map to the same vertex.     

One sees that the vertex $v_1$ mapped to $1$ could be reconfigured to $2$;  this means that the shown homomorphism is adjacent to the homomorphism we get from it by remapping $v_1$ to the vertex $2$. 
In fact, it will follow from Lemma \ref{lem:pushup} that one can verify this simply by verifying that the resulting remapping is a homomorphism.
Very few other single vertices can be reconfigured, $v_8$ can be reconfigured up from $1$ to $2$ and $v_6$ down from $1$ to $0$.  The vertex $v_4$ can be configured up, and $v_{10}$ down.   After some reductions, we will mostly talk about moving vertices up and ignore movements down.  After $v_1$ moves up to $2$, it could then move to $3$, then $v_2$ could move  up to $3$ allowing $v_3$ and $v_4$ to move up to $0$ together, though neither can move on its own.  With some patience one can show, greedily, that we can keep moving vertices up-- the reconfiguration graph has large cycles. 

Our main theorem allows us to discover this without the need for such patience.  The cycle $D$ from the example will be represented by the orientation string $+-+-$, defined in the next section, and its primitive root $\sqrt{D}$, defined just before the statement of Theorem \ref{thm:wind_components}, will be $+-$. We will write $D = \sqrt{D}^r$ where $r = 2$.  As the map in the figure has wind $w = 2$, its wind around $\sqrt{D}$ will be $r\cdot w = 4$.  The cycle $C$ is described by the orientation string $+-+-+--++--++-$ from which we can remove symbols to get $+-+-+-+-+-$, which is $\sqrt{D}^R$ where $R = 5$. As this $R = 5$ is greater than $r \cdot w = 4$, we have $w < R/r$, which puts us in case (2) of Theorem \ref{thm:wind_components}. This tells us that the wind $2$ subgraph $\Hom_2(C,D)$ of $\Hom(C,D)$ is cyclic, meaning that from any map $\phi$ we can get back to $\phi$ in $\Hom_2(C,D)$ by a non-trivial cycle of reconfigurations up.   
\end{example}

\subsection*{Background Definitions and Results}

 As is observed in Fact 2.2 of \cite{BLS18}, a reflexive digraph $D$ of length $3$ containing a transitive triangle is {\em contractible} and so $\Hom(G,D)$ is connected for all $G$. Thus $\Recon(D)$ is trivial-- all instances are YES instances. We thus restrict our attention to the case that $D$ is {\em non-contractible}: it has length at least $4$ or is a directed $3$-cycle.

 Denoting a digraph cycle $C$ as  $C  = c_0c_1\dots c_{m-1}c_0$ specifies that its vertex set is $\{c_0, \dots, c_{m-1}\}$ and that $v_iv_{i+1}$ is an edge for $i =0, \dots, m-1$.
 A cycle is assumed to have an underlying orientation in the direction of increase of the indices of the vertex labels.
 For the particular cycle  $D = (0)(1) \dots (n-1)(0)$ of length $n$, we use integers modulo $n$ as the vertex set rather than as indices of the vertices. 
 
 With respect to this underlying ordering, the edge $c_ic_{i+1}$ is {\em forward} if $c_i \to c_{i+1}$, {\em backward} if $c_i \ot c_{i+1}$, or {\em symmetric} if $c_i \to c_{i+1}$ and $c_i \ot c_{i+1}$.  The {\em algebraic length} of an oriented cycle $C$ is the number of forward edges minus the number of backward edges. If a cycle  has any symmetric edges, the algebraic length is not defined, but it contains cycles of various algebraic lengths.  
   For a homomorphism $\phi$ of a cycle $C = c_0 \dots c_{m-1}c_0$ to a cycle $D = (0)(1) \dots (n-1)(0)$ an edge $c_ic_{i+1}$ is   {\em increasing, stationary} or {\em decreasing} under $\phi$ according to whether $\phi(c_{i}) - \phi(c_{i+1})$ is $-1,0$ or $1$. The {\em increase} of the homomorphism is the number of increasing edges  minus the number of decreasing edges.  

\begin{definition}\label{def:wind}
The {\em wind} of a homomorphism $\phi: C \to D$ of digraph cycles is the increase of $\phi$ divided by the length of $D$.
\end{definition}

The wind of a map $\phi: C \to D$,  which is clearly an integer,  depends on the underlying orientation of $C$ and $D$.  Reversing the orientation of  either $C$ or $D$ multiplies the wind by $-1$. 
The following fact, which is easy to check, is given (in a bit more generality) as Lemma 3.4 of  \cite{BLS21}.         
\begin{fact}\label{WindInvariance}
 For reflexive digraph cycles $C$ and $D$ where $D$ is non-contractible, the wind of maps of  $\Hom(C,D)$ is constant over components.
 \end{fact}

 As $C$ is reflexive, we have for any arc $\phi\phi'$ of $\Hom(C,D)$ and any vertex $c_i$ of $C$ that $\phi(c_i)\phi'(c_i)$ is an arc of $D$,  and so $\phi'(c_i)$ is in $\{ \phi(c_i) , \phi(c_i) \pm 1 \}$.  The vertex $c_i$ {\em moves up} via $\phi\phi'$ if  $\phi'(c_i) = \phi(c_i) + 1$ and {\em moves down} if $\phi'(c_i) = \phi(c_i) - 1$. The edge $\phi\phi'$ is an {\em up edge} if all vertices that move, move up, and is a {\em down edge} if all vertices that move, move down.  Clearly, if $\phi\phi'$ is an up edge, then $\phi'\phi$ is a down edge.  An edge of  $\Hom(C,D)$ being a forward, backward or symmetric edge should not be confused with it being an up or down edge. 

  A component of $\Hom(C,H)$ is {\em cyclic} if for any two maps $\phi$ and $\phi'$ in the component, there is a path of up edges from $\phi$ to $\phi'$.  
 We call the component `cyclic' because in such a component one can get from $\phi$ back to $\phi$ by a non-trivial path of up edges. 
The single component in part (1) of Fact \ref{fact:mainsym} is a cyclic component.

The following adapts notation and ideas that were developed for paths in \cite{LLP24} to talk about  homomorphisms between digraph paths and to prove the so-called S\l upeckiness of all non-contractible reflexive cycles.   

 A cycle $C = c_0 \dots c_{m-1} c_0$ can be represented by its {\em orientation string} $x_1\dots x_m$. This is the string of length $m$ over the alphabet $\{-,+,*\}$ whose $i^{th}$ letter $x_i$ is $-, +,$ or $*$  depending on whether the $i^{th}$ edge $c_{i-1}c_{i}$ is a backward arc, a forward arc, or a symmetric arc. 
So that this orientation string uniquely represents $C$, we must think of $C$ not only with an underlying orientation, but also as a {\em pointed cycle} $C$ with base-point $c_0$.  In particular, we distinguish $C$ from its  {\em $i^{th}$ shift} $\sigma^i(C) = c_ic_{i+1} \dots c_{i-1}c_i$ which has base-point $c_i$.

  For example, a symmetric $4$-cycle is denoted $ \ast \ast \ast\, \ast$, a forward directed $4$-cycle is  $++++$, and the oriented pointed $4$-cycles $+-+-$ and $-+-+$ are called alternating $4$-cycles; they are the same as cycles, but as pointed cycles, are shifts of each other.

The authors of \cite{LLP24} used a partial ordering of orientation strings to give a useful description of homomorphisms between pointed paths;  we adapt this to pointed cycles.

\cite{Siggers}

Define a partial ordering on the set of all finite orientation strings as follows. For strings $D = y_1 y_2 \dots  y_n$ and $C = x_1 x_2 \dots x_m$, with $n \leq m$ let $D \leq^* C$, and call $D$ a {\em $*$-substring} of $C$,  if there is a strictly increasing function
  \[ \alpha = (\alpha(1), \dots, \alpha(n)): [n] \to  [m], \]
 called a {\em selection function}, such that 
we can get $D$ from the substring $x_{\alpha(1)}\dots x_{\alpha(n)}$ of $C$ by possibly changing letters to $*$. This implies, in particular, that $x_{\alpha(i)}$ is the same letter as $y_i$ unless $y_i$ is a $*$.      

For example, we would have that 
  \[ **+-* \leq^* -++-- \leq^* -++--+-; \]
the selection function for the first inequality is $(1,2,3,4,5)$ while for the 
second equality, there are three different selection functions showing that $-++--$ is a $*$-substring of $-++--+-$; they are $(1,2,3,4,5)$, $(1,2,3,4,7)$, and $(1,2,3,5,7)$.

A homomorphism of cycles $\phi: C \to D$ is {\em increasing} if every edge of $C$ under $\phi$ is increasing or stationary, but not all are stationary,  it is {\em decreasing} if every edge is decreasing; it is {\em monotone} if it is increasing or decreasing.  

The following should be clear and comes immediately from an analogous result about pointed paths in \cite{LLP24}: if $\alpha$ is the selection function that finds $D$ as a $*$-substring $x_{\alpha(1)}x_{\alpha(2)} \dots x_{\alpha(n)}$ of $C$, then there is a monotone wind $1$ homomorphism of $C$ to $D$ in which the increasing edges are exactly those edges $c_{(\alpha(i)-1)}c_{\alpha(i)}$ selected by $\alpha$.
\begin{fact}
  There is wind $1$  monotone homomorphism $\phi: C \to D$ of cycles with $\phi(c_0) = 0$ if and only if $D$ is a $*$-substring of $C$.    
\end{fact}

We state an obvious extension of this to account for wind, and for homomorphisms taking $c_0$ to different vertices of $D$. 
For orientation strings $D = y_1 \dots y_n$ and $Z = z_1 \dots z_p$, $DZ$ is the concatenation $y_1 \dots y_nz_1 \dots z_p$ of $D$ and $Z$, $D^w$ for positive integer $w$ is the concatenation of $w$ copies of $D$. 
\begin{fact}
  There is wind $w$  monotone homomorphism $\phi: C \to D$ of cycles with $\phi(c_0) = i$ if and only if $\sigma^i(D^w)$ is a $*$-substring of $C$; ie, if and only if  $\sigma^i(D^w) \leq^* C$.  
\end{fact}

 \begin{example}\label{ex1}
     Where $C = x_1 \dots x_7 = -++--+-$ and $D =  -++--$ there are four wind $1$ monotone homomorphisms of $C$ to $D$.
     They are shown in \Cref{fig:ex1}. 
\begin{figure}
 \[   \begin{array}{c|ccccccc|c}
\multirow{2}{*}{C} & x_1 & x_2 & x_3 & x_4 & x_5 & x_6 & x_7 &\multirow{2}{*}{\mbox{ Selection Function }}\\\cline{2-8}
             & -  & +  & +  & - & - & + & - & \\\hline
   \phi_1 & -  & +  & +  &   & -  &    & - & (1,2,3,5,7)  \\
   \phi_2 & -  & + & +  & -  &    &    & - & (1,2,3,4,7) \\
   \phi_3 & -  & +  & +  & - & - &  &  &  (1,2,3,4,5) \\
   \phi_4 &   & +  & +  & - & - &  & - & (2,3,4,5,7) \\
\end{array} \]\caption{Homomorphisms of $C = -++--+-$ to  $D =  -++--$}\label{fig:ex1}
\end{figure}
For the first three, we find $D$ as a $*$-substring of $C$, so map $c_0$ to $0$.  
 The selection function for $\phi_1$ is $\alpha = (1,2,3,5,7)$. To determine where, say, the vertex $c_5$ maps, we observe that four of the edges that come between it and $c_0$ are increasing, the edges $x_1, x_2,x_3$ and $x_5$ are, so $c_5$ maps to $4$. 

For the homomorphism $\phi_4$, we find $\sigma^1(D) = ++---$ as a $*$-substring 
of $C$.  This maps $c_0$ to $1$-- it maps $x_6$ to $0$, $x_7$ to $1$, $x_0$ to $1$, $x_2$ to $1$, $x_2$ to $2$, etc. The selection function $\alpha$ for $\phi_4$ is the selection function of $\sigma^1(D)$ as a $*$-substring of $C$, not of $D$, so it is 
$\alpha = (2,3,4,5,7)$.   
\end{example}

\subsection*{Statements of Results}
 
With a couple of new definitions, we can state our main theorem. 
Representing a cycle $D$ by its orientation string, the {\em primitive root $\sqrt{D}$} of $D$ is the shortest substring such that $\sqrt{D}^r = D$ for some integer $r$.  A cycle $D$ is {\em primitive} if $\sqrt{D} = D$.  Finding the primitive root of a cycle $D$ amounts to finding the minimum $i \geq 1$ such that $\sigma^i(D) = D$, so this can be done in time $O(n^{3/2})$ where $|D| = n$, (see \Cref{lem:Complexity}).

Recall that  $\Hom_w(C,D)$ is the subgraph of $\Hom(C,D)$ induced by maps of wind $w$.   We will only consider reflexive cycles $D$ of girth at least $4$, so \Cref{WindInvariance} applies to say that wind is preserved over components of $\Hom(C,D)$, so to describe them, it is enough to describe the components of the subgraphs $\Hom_w(C,D)$.

The main result that we prove in this paper is the following. 
\begin{theorem}\label{thm:wind_components}
Let $C$ and $D = \sqrt{D}^r = (y_1\dots y_s)^r$ be reflexive digraph cycles such that $D$ is non-contractible.  
Let $R$ be the maximum value such that $\sigma^i(\sqrt{D}^R) \leq^* C$ for some $i$.
Except in the exceptional case that $D$ is a symmetric cycle and $C$ is a directed cycle, the  subgraph $\Hom_w(C,D)$ of $\Hom(C,D)$, for $w \geq 0$, consists of 
\begin{enumerate}
   \item a single cyclic component containing a copy of $D$ if $w = 0$, 
  \item a single cyclic component if $0 <  w < R/r $, or if $ 0 < w = R/r$ and $\sigma^i(\sqrt{D})^{wr}y_{i+1} \leq^* C$ 
           for all shifts $\sigma^i$ of $\sqrt{D}$,   
  \item $c$ non-cyclic components if $0 < w = R/r$ and there are $c$ values of $i \in [s]$ for which  $\sigma^i (\sqrt{D})^{wr}y_{i+1} \nleq^* C$,
  \end{enumerate}  
 and nothing if $ R/r < w$. In the exceptional case, (2) and (3) are replaced with: a single cyclic component if $0 < w \leq^* R/r$. 
   \end{theorem}

 \begin{remark}\label{1isenough}
    As $D$ is non-contractible, a map in $\Hom_w(C,D)$, for $w \geq 1$, can be viewed as a map in $\Hom_1(C,wD)$. Doing this does not change $D$ or $R$ but replaces $r$ with $wr$, and statements (2) and (3) remain the same.   Thus it is enough to prove the theorem in the cases $w = 0$ and $w=1$.  
  \end{remark}

 Again, applying Theorem \ref{thm:wind_components} to the reverse  $C^{-1}$ of $C$ we also get a characterisation of the components of $\Hom(C,D)$ of negative wind, so this gives a comprehensive description of the components of $\Hom(C,D)$.  Fact \ref{fact:mainsym} now follows by taking $C=*^m$ and $D=*^n$, so $Y = *$ and $(r,s) = (n,1)$. 

In Section \ref{sec:tools} we recall results from \cite{BLS21}, \cite{LLP24}, and \cite{MZ12} that will allow us, among other things, to reduce the connectivity of $\Hom_1(C,D)$ to that of the subgraph induced on monotone homomorphisms.  Using these tools, we then prove Theorem \ref{thm:wind_components} in Section \ref{sec:proof}.
In Section \ref{sec:Algorithms} we give simple algorithms to determine the primitive root of $D$ and to determine which of the conditions hold in Theorem \ref{thm:wind_components}.  From this, we get, in \Cref{Complexity}, that the problem $\Recon(D)$ for reflexive digraph cycles can be solved for cycle instances in polynomial time and logarithmic space.  

\begin{remark}\label{rmk:looplessC}
  It was shown in \cite{BLS21}, that unless the target $D$ is a digraph $4$-cycle containing a $4$-cycle of algebraic length $0$, the presence of loops on an instance $C$ does not change the existence of a path.   
  From this we get that Theorem \ref{thm:wind_components} holds for general cycle instances $C$ except in the case that $D$ contains a $4$-cycle of algebraic girth $0$.    
  In the case that $D$ is a reflexive digraph of length $4$ containing a $4$-cycle of algebraic length $4$, the complexity of $\Recon(D)$ is still unknown. 
\end{remark}

 \section{Tools for reducing to monotone homomorphisms}\label{sec:tools}

 In this section we recall known results, and tailor from them several lemmas, Lemmas \ref{lem:pushup}, \ref{uptomono} and \ref{lem:refine},  that will allow us to prove  Theorem \ref{thm:wind_components}  in the next section. The first two will allow us to restrict our attention mostly to monotone homomorphisms, and the third will allow us to make assumptions about the edges of $\Hom(C,D)$ between monotone homomorphism. 
   In all these lemmas $C =  c_0c_1 \dots c_{m-1}c_0$ and $D = (0)(1) \dots (n-1)(0)$ are digraph cycles, $D$ is non-contractible, and $0 \leq w \leq m/n$.    

  
\begin{lemma}\label{lem:pushup}
 Let $\phi \in \Hom(C,D)$ and let $S$ be the vertices of a subpath of $C$ consisting of 
 edges that are stationary under $\phi$, so $\phi$ maps $S$ to a single vertex $d$ of $D$. 
 If the map $\phi'$ we get from $\phi$ by moving the elements of $S$ up to $d+1$ is a homomorphism, then it is adjacent to $\phi$ in $\Hom(C,D)$.  We call the edge in $\phi\phi'$ a {\em one-step up edge}.
\end{lemma}
\begin{proof}
With the setup of the lemma, assume, at first, that $d \to d+1$ in $D$, and let $c \to c'$ in $C$. 
 We show that  $\phi(c) \to \phi'(c')$.  
 If $c'$ is not in $S$, then $\phi(c) \to \phi(c') = \phi'(c')$ because $\phi$ is a homomorphism and $c'$ does not move. So we may assume that $c'$ is in $S$. 
 If $c$ is also in $S$, then $\phi(c) = d \to d+1 = \phi'(c')$. So we may assume that $c \not\in S$.  Then $\phi(c) = \phi'(c) \to \phi'(c')$.  

 If, on the other hand, $d \ot d+1$, then let $c \ot c'$ in $C$. The same argument, flipping arrows, shows that $\phi(c) \ot \phi'(c')$.  
 Either way, we get that $\phi$ is adjacent to $\phi'$ in $\Hom(C,D)$, as needed. 
 \end{proof}

 For a homomorphism $\phi \in \Hom(C,D)$, a subpath $P =c_a\dots c _b$ of $C$ is a {\em cutback} if its increase is $0$ (so $\phi(c_a) = \phi(c_b)$)
 and the increase of $c_a \dots c_i$ is negative  for  all $i \in \{a+1, \dots, b-1\}$.   In Figure \ref{fig:1}, the path $v_0v_1v_2$ is a cutback, 
 to make the shown map monotone, we will want to push $v_1$ up to where $v_0$ and $v_2$ are.  This is what the next lemma lets us do. 


\begin{lemma}\label{uptomono}
 From any homomorphism $\phi$ in $\Hom(C,D)$, and any cutback $P =c_a\dots c _b$ of $C$,  there is a path of up edges from $\phi$ to the homomorphism
 $\phi'$ we get from $\phi$ by setting $\phi'(c_i) = \phi(c_a)$ for all $c_i \in P$.  
\end{lemma}
\begin{proof}
 The proof is by induction on $m$, the minimum, over $i \in \{a, \dots, b\}$ of the increase of the subpath $c_a \dots c_i$ of $P$ under $\phi$. 
 If $m = 1$, then $\phi$ takes all of $c_{a+1} \dots c_{b-1}$ to $\phi(c_a) - 1$, and as $\phi(c_a)$ has a loop, we can apply \Cref{lem:pushup} to get a path 
 up to $\phi'$.  So assume that $m \geq 1$.  Then where $a'$ is the first index in $\{a, \dots, b\}$ for which $\phi(a') = \phi(a)-1$ and $b'$ is the last, we have that 
 $c_{a'} \dots c_{b'}$ is a cutback with smaller $m$. By induction, we can move all  vertices in this path up to $\phi(a')$, getting a cutback with $m = 1$, and then by the $m = 1$ case, we can get a path from this up to $\phi'$. 
\end{proof}

 Any non-monotone homomorphism has a cutback, and so we immediately get the following. 

\begin{corollary}\label{uptomono2}
 From any homomorphism $\phi$ in $\Hom(C,D)$,  there is a path of up edges to a monotone  homomorphism.
\end{corollary}

In fact, we can get such a path by \Cref{uptomono} by only pushing up cutbacks, and  there is a unique monotone map we get in this way; call it the {\em monotone push up of $\phi$}. 


   Our last main tool, a complement to Lemma \ref{lem:pushup}, will allow us to say that,  except in the exceptional case, the only edges of $\Hom(C,D)$ that we will have to consider are one-step up edges between monotone maps.    We use results from \cite{MZ12} and \cite{BLS21} for this.  For an edge $\phi\phi'$ of $\Hom(C,D)$, let $\Neq(\phi,\phi')$ be the set of vertices $c$ of $C$ for which 
   $\phi(c) \neq \phi'(c)$. For a subset $T \subset \Neq(\phi,\phi')$, let $\phi_T$ be the map that agrees with $\phi'$ on $T$ and with $\phi$ on $V(C) \setminus T$.  
   It was shown in \cite{MZ12} that if $\phi_T$ is a homomorphism, then $\phi\phi_T\phi'$ is a path in $\Hom(C,D)$. The edge $\phi\phi'$ is {\em non-refinable} if there is no 
  $T \subset \Neq(\phi,\phi')$ such that $\phi_T$ is a homomorphism.

  The following is Lemma 2.10 of \cite{MZ12}, the `moreover' part is not in the statement of the lemma, but is in the proof. A {\em strong component} in a digraph is a subgraph that is maximal with respect to the property that we can get from any vertex to any other by a forward directed path. It is {\em terminal} if it has no out arcs to vertices not in $T$.    

  \begin{lemma}[\cite{MZ12}]\label{MZ}
    For digraphs $G$ and $H$ and an edge $\phi\phi'$ of $\Hom(G,H)$ let $A$ be the digraph on the vertices of $G$ such that  $g \to_A g'$ if $\phi(g) \to \phi'(g')$.
    The edge $\phi\phi'$ is non-refinable if and only if $A$ is strongly connected; moreover, if $T$ is a terminal strong component of $A$, then $\phi\phi_T\phi'$ 
    is a path in $\Hom(G,H)$.        
  \end{lemma}    
  
   The same construction was used in \cite{BLS21} in the case that $D$ was a cycle,  only the word `indecomposable' was used instead of `non-refinable'\footnote{Though referencing \cite{MZ12} the authors of \cite{BLS21} apparently did not read it all. Sorry!} 
   The following is Lemma 5.2 of \cite{BLS21}.
 
  \begin{lemma}[\cite{BLS21}]
     Any non-refinable edge of $\Hom(G,D)$ is either an up edge or a down edge. 
  \end{lemma} 

  This allows us to consider only the non-refinable up and down edges in $\Hom(C,D)$.    
  The following describes the non-refinable edges of monotone maps. 

\begin{lemma}\label{lem:refine}
   Let $\phi$ be a monotone map, and $\phi\phi'$ be a non-refinable up edge of $\Hom(C,D)$ for reflexive digraph cycles $C$ and $D$ where $D$ has length $n$. 
   Let $T$ be the set of vertices that it moves up.
    \begin{enumerate}
     \item If $C$ has length $m =  n$ then $C$ is a directed cycle, $D$ is a symmetric cycle, and  $T = V(C)$. If $\phi \to \phi'$ then $C$ is backwards directed, and if 
             $\phi \ot \phi'$ then $C$ is forwards directed. 
     \item If $C$ has length $m >  n$ then $T$ is a single vertex.  
       \end{enumerate}  
    \end{lemma}
  \begin{proof}
       First, we assume that $\phi \to \phi'$ is a non-refinable up edge of $\Hom(C,D)$.   By  Lemma \ref{MZ}, $T$ is a terminal
       strong component of the auxiliary digraph $A$, and $\phi_T = \phi'$.  We start with two observations, the first is about adjacent homomorphisms, 
       the second is about $A$.  

       \begin{claim}
         If $T$ contain both endpoints of an increasing edge $c_ic_{i+1}$ of $C$, then $c_ic_{i+1}$ is a backwards (non-symmetric) edge of $C$, and 
         $\phi(c_{i+1})\phi(c_{i+1})+1$ is a symmetric edge of $D$.  
       \end{claim}
       \begin{proof}\claimproof
         If $c_i \to c_{i+1}$ then $\phi(c_i) \to \phi'(c_{i+1}) = \phi(c_i) + 2$, which is impossible as $D$ contains no transitive triangle, so $c_i \ot c_{i+1}$.
        Now as $c_i \ot c_{i+1}$ we get $\phi(c_{i+1}) = \phi'(c_i) \ot \phi'(c_{i+1}) = \phi(c_{i+1}) + 1$ because $\phi'$ is a homomorphism, and as
        $c_{i+1} \to c_{i+1}$ we get $\phi(c_{i+1}) \to \phi'(c_{i+1}) = \phi(c_{i+1}) + 1$. 
       \end{proof}

       \begin{claim}
         $A$ has no symmetric edges, so $T$ is either a backwards directed cycle or a single vertex. 
       \end{claim}
       \begin{proof} \claimproof
         An increasing edge $c_ic_{i+1}$ of $C$ clearly becomes a backwards edge in $A$-- as $D$ contains no transitive triangles, we must have $c_i \ot_A c_{i+1}$ and as 
         $D$ is reflexive we have $c_{i+i} \not\to_A c_{i+1}$.

        On the other hand, for a stationary edge $c_ic_{i+1}$ of $C$ mapped to a vertex $i$ of $D$ we have  $c_i \not\sim c_{i+1}$ in $A$ if  $i(i+1)$ is symmetric in $D$, and otherwise $c_i \to_A c_{i+1}$ or $c_i \ot_A c_{i+1}$, but not both. 

         As $T$ is a terminal strong component in an oriented cycle, it is either a directed cycle, or a single vertex. If it is a directed cycle, then it is a backwards cycle, 
         $C$ has increasing edge, and so $A$ has backwards edges. 
        \end{proof}

         If $C$ has length $m = n$, then all edges of $C$ are increasing edges, and all edges of $A$ are backwards, and so $T = V(A) = V(C)$.
         By the first claim, we get  that all edges of $C$ are backwards edges, and all edges of $D$ are symmetric. This gives statement (1).  
         
         So we may assume that $C$ has length $m >  n$.  By the second claim, we are done if we can show that $A$ is not a directed cycle, but for this 
         it is enough to show that there is some edge $c_ic_{i+1}$ for which $c_i \not\ot_A c_{i+1}$.  Assume, towards contradiction, that $A$ is a directed cycle. 
         Then by the first claim we again have that $D$ is symmetric. Now, however, there is a stationary edge, so there is some vertex $c_i$  such that $c_ic_{i+1}$ is increasing, but 
         $c_{i-1}c_i$ is stationary.  Then $c_i$ has no out edges in $A$. As $D$ is reflexive $\phi(c_{i+1}) \to \phi(c_i)+1$ so $c_i \not\to_A c_{i+1}$, and as $\phi(c_i)\phi(c_i)+1$ is symmetric  $\phi(c_{i-1}) \to \phi(c_i) + 1$, so $c_i \not\to_A c_{i-1}$.
  \end{proof}

 \section{Proof of the main theorem via monotone homomorphisms}\label{sec:proof}

   Let $\Mon(C,D)$ be the subgraph of  $\Hom(C,D)$ induced on the set of monotone maps, and for wind $w$ let $\Mon_w(C,D)$ be the subgraph of $\Hom_w(C,D)$ induced  by monotone maps.  By \Cref{uptomono}, every map in $\Hom(C,D)$ is in the same component as a map in $\Mon(C,D)$, so we start by understanding the components of $\Mon(C,D)$.  

 The case $w = 0$ is easy. 
  
  \begin{fact} 
    The graph $\Mon_0(C,D)$ is isomorphic to $D$. 
   \end{fact} 
   \begin{proof}
          The vertices of $\Mon_0(C,D)$ are clearly the maps $\phi_i$ for all $i \in V(D)$, that map all vertices of $C$ to the vertex $i$. 
          Assume that $i \to i+1$ in $D$, we show that $\phi_i \to \phi_{i+1}$.   Indeed, for any arc $c \to c'$ of $C$ we have 
          $\phi_i(c) = c \to c' = \phi_{i+1}(c')$, so $\phi_i \to \phi_{i+1}$.  Similarly, if $i \ot i+1$, then $\phi_i(c') \ot \phi_{i+1}(c)$, so $\phi_i \ot \phi_{i+1}$ 
   \end{proof}

  As observed in \Cref{1isenough}, for the case of $w \geq 1$, it is enough to consider the case of $w = 1$.  
  Recall that denoting (pointed) cycles $C$ and $D$ by their orientation strings
           \[ C = x_1x_2 \dots x_m \qquad D = y_1y_2 \dots y_n\]
  homomorphisms $\phi \in \Mon_1(C,D)$ are represented by their selection functions $\alpha_\phi: [n] \to [m]$.  For each $i \in [n]$,  let $\Mon_1(C,D;i)$ be the subgraph of $\Mon_1(C,D)$ induced by vertices $\phi$ such that $\phi(c_0) = i$. So $\Mon_1(C,D;i)$ consists of the different copies of $\sigma^i(D)$ as $*$-substrings of $C$.

\begin{example}
 Referring to \Cref{ex1}, and so \Cref{fig:ex1}, the first three maps are in $\Mon_1(C,D;0)$, while the fourth is in $\Mon_1(C,D;1)$.
 Notice how we get from $\phi_1$ to $\phi_2$ to $\phi_3$ to $\phi_4$ by moving letters of $D$ as a $*$-substring down.  
 The following discussion explains how this gives a path $\phi_1\phi_2\phi_3\phi_4$ of one-step up edges in $\Mon_1(C,D)$.
 \end{example}

 For $i \in [n]$, we order $\Mon_1(C,D;i)$ by setting $\phi \geq \phi'$ if the selection functions satisfy $\alpha_\phi(i) \leq \alpha_{\phi'}(i)$ for all $i$. The reversal of the
 order is intentional and will be explained presently. It is clear that $\Mon_1(C,D;i)$ has minimum and maximum elements with respect to this ordering; we call them $\Phi_i^m$ and $\Phi_i^M$ respectively.

\begin{example}
 Referring to \Cref{ex1}, we have $\Phi_0^m = \phi_1 <  \phi_2 < \phi_3 = \Phi_1^M$.
 \end{example}

  A {\em monotone one-step up edge} is a one-step up edge $\phi\phi'$ in $\Mon_1(C,D)$. It moves the vertices $S$ of some path $c_a \dots c_{b-1}$ of stationary edges 
  from some $d-1$  to $d$. As it is monotone, we know that $\phi(c_b) = d$.   If $c_0$ is not in this path, then $\phi$ and $\phi'$ are both in 
   $\Mon_1(C,D;i)$ for some $i$, and the selection functions  $\alpha_\phi$ and $\alpha_{\phi'}$ are the same except that $\alpha_\phi(d-i) = b$ and $\alpha_{\phi'}(d-i) = a$.
   That is, a monotone one-step up edge from $\phi$ corresponds to moving one value of the selection function $\alpha_\phi$ down  past values not in the image of $\alpha_\phi$.   For selection functions $\alpha, \alpha': [m] \to [n]$ with $\alpha(i) \leq \alpha'(i)$ for all $i$, we can clearly move $\alpha'$ down to $\alpha$ in this way, one index at a time, starting with the lowest $i$ on which they differ.   Thus for $\phi$ and $\phi'$ in $\Mon_1(C,D;i)$,  we have that $\phi' \leq \phi$, if and only if there is a sequence of monotone one-step up edges in $\Mon_1(C,D;i)$ from $\phi'$ to $\phi$. This is why we reversed the ordering.

   If $c_0$ is in the path $c_a \dots c_{b-1}$, then the edge is from $\Mon_1(C,D;i)$, for some $i$, to $\Mon_1(C,D;i+1)$.  In this case, $\alpha_\phi(d-i)$ moves from $b$ down 
   past $0$ modulo $m$ to some $a$ that was above all other edges selected by $\alpha_\phi$.  In particular, the selection function $\alpha_\Phi$ for $\Phi_i^M$ 
   finds the left-most copy of $\sigma^i(D)$ as a $*$-substring of $C$, and so has a one-step up edge to $\Mon_1(C,D;i+1)$ if and only if $\sigma^i(D)y_i$ is a 
   $*$-substring of $C$.

 \begin{example} 
 Referring again to \Cref{ex1}, there is a path $\phi_1\phi_2\phi_3\phi_4$ of monotone one-step up edges in $\Mon_1(C,D)$. 
  In fact, in each of these edges, we move a `$-$' down , which means vertices are moved up past a backwards edge, so 
  this path is actually  $\phi_1 \ot \phi_2 \ot \phi_3 \ot \phi_4$. 
  \end{example}

   Summarizing this discussion about $\Mon_1(C,D;i)$, we have the following facts.

 \begin{fact}
     For each $i \in [n]$, the graph $\Mon_1(C,D;i)$ is connected.  There are elements $\Phi_i^m$ and $\Phi_i^M$ such that for every element $\phi$ there is a 
    path of monotone one-step up edges from $\Phi_i^m$ to $\phi$ to $\Phi_i^M$.  
   \end{fact}
   The discussion about edges between $\Mon_1(C,D;i)$ and $\Mon_1(C,D;i+1)$ is as follows.

  \begin{fact}\label{MonRedEasyWay}  
       There is a monotone one-step up edge from $\Phi_i^M$ in $\Mon_1(C,D;i)$ to some map in $\Mon_1(C,D;i+1)$ if and only $\sigma^i(D)y_i \leq^* C$.  
  \end{fact}

    We will strengthen this to say that, except in the exceptional case,  there is an up edge from $\Mon_1(C,D;i)$ to $\Mon_1(C,D;i+1)$ if and only
     if $\sigma^i(D)y_i \leq^* C$; but as we will need something a bit stronger than this, we prove it all at once.

    For any map $i$, let $\Mon_1^+(C,D;i)$ be the set of maps in $\Hom_1(C,D)$ whose push-up (defined after \Cref{uptomono2}) is in $\Mon_1(C,D;i)$. The graphs $\Mon_1^+(C,D;i)$ over all $i$ partition the vertices of $\Hom_1(C,D)$.   
    An {\em up path} $\phi_1 \dots \phi_\ell$ in $\Hom_1(C,D)$ is a path under which every vertex moves up-- by this we mean that for all vertices $c_i$ of $C$ the increase of the  path $\phi_1(c_i)\phi_2(c_i) \dots \phi(c_\ell)$ is non-negative.   A path of up edges is an up path, but the converse is not necessarily true.     

    \begin{proposition}\label{MonReduction}
       Let $C$ and $D$ be digraphs with $m > n$.   There is  an up edge out of $\Mon_1^+(C,D;i)$ 
       if and only if $\sigma^i(D)y_i \leq^* C$.
       Consequently, there is an up edge out of $\Mon_1(C,D;i)$ 
          if and only if $\sigma^i(D)y_i \leq^* C$.
    \end{proposition}
    \begin{proof}
      The `if' direction is immediate from the previous fact, as an up edge out of $\Mon_1(C,D;i)$ must clearly go to $\Mon_1(C,D;i+1)$. 
      For the other direction, assume that there is an up edge $\phi\phi'$ from $\Mon_1^+(C,D;i)$ to $\Mon_1^+(C,D;i+1)$. As $m > n$ we can refine this edge, 
      by Lemma \ref{lem:refine}, into a path of non-refinable edges; there is a last vertex on the path in $\Mon_1^+(C,D;i)$, and so we may assume that
      $\phi\phi'$ moves a single vertex $c_a$, and must move it up, or $\phi'$ would also be in $\Mon_1^+(C,D;i)$.
       Let $d = \phi(c_a)$, so $\phi'(c_a) = d+1$.  

       In the following arguments, we use `$\leq$' to compare values in $\{d-1, d, d+1\} \subset V(D)$, so it is well-defined by 
       $d-1 \leq d \leq d+1$.  
       Let $\Phi = \Phi_M(i)$. As $\phi$ is in $\Mon_1^+(C,D;i)$ but $\phi'$ is not,  we have  $\phi(c_j) \leq \Phi(c_j)$ for all $j \in \{ a-1, a, a+1 \}$, 
       but this is not true of $\phi'$ in place of $\phi$, and so $\phi(c_a) \leq \Phi(c_a) < \phi'(c_a)$ which means that $\Phi(c_a) = d$.

       \begin{claim} $\Phi$ is adjacent to the map $\Phi'$ we get by moving $c_a$ up to $d+1$. 
       \end{claim}
       \begin{proof}\claimproof
             By \Cref{lem:pushup}, it is enough to show that $\Phi'$ is a homomorphism, and so is  enough to show that it maps the edges $e_{a-1}=c_{a-1}c_a$ and $e_a = c_ac_{a+1}$ to edges of $D$.  It maps $e_{a-1}$ to an edge as it maps it to the same place as $\phi'$ does. Indeed, both map $c_a$ to $d+1$; and $c_{a-1} \sim c_a$ we have $\phi'(c_{a-1}) \geq \phi'(c_a) - 1 =d$ and as $\Phi$ is monotone we have $\phi'(c_{a-1}) = \phi(c_{a-1}) \leq \Phi(c_{a-1}) \leq \Phi(c_a) = d$.

             Similar considerations give  $d \leq \phi(c_{a+1} = \phi'(c_{a+1}) \leq  \Phi'(c_{a+1}) \leq d+1$.  If $\Phi'(c_{a+1}) = d$ then $\Phi'$ is the same on  $c_ac_{a+1}$ as $\phi'$, and if $\Phi'(c_{a+1}) = d+1$, then $\Phi'$ maps $c_ac_{a+1}$ to a loop.
             \end{proof}

        The claim shows that $\Phi$ has an edge up to $\Phi'$. Where $c_b$ was the first vertex of $C$ above $c_a$ with $\Phi(b) = d+1$, we get the monotone push-up $\Phi''$ of 
       $\Phi'$ by moving  $c_{a+1} \dots c_{b-1}$ from $d$ up to $d+1$.  As this is a homomorphism, and we get it from $\Phi$ by moving up $c_a \dots c_{b-1}$, 
        we have by \Cref{lem:pushup} that $\Phi\Phi''$ is a  monotone one-step up edge. As $\Phi$ is the maximum map in $\Mon_1(C,D;i)$ we have that $\Phi''$ is in $\Mon_1(C,D;i+1)$, and  so by the previous fact we have $\sigma^i(D)y_i \leq^* C$, as needed. 
        \end{proof}

    With all the bits in order, we are now ready to prove \Cref{thm:wind_components}. 

     \begin{proof}
         Let $C$ and $D = \sqrt{D}^r$ be reflexive digraph cycles such that $D$ is non-contractible, and let $R$ be the maximum integer for which
         $\sigma^i(\sqrt{D}^R) \leq^* C$ for some $i$.  
 
         If $w = 0$ then by Fact 3.1 $\Mon_0(C,D)$ is $D$, and by \Cref{lem:pushup} everything in $\Hom_0(C,D)$ has a path up and down to $\Mon_0(C,D)$, so is 
         connected and cyclic. Thus part (1) of the theorem is proved. 

         As  $\Hom_w(C,D) = \Hom_1(C,wD)$  when $w \geq 2$, we may therefore assume that $w =1$. 
         prove the result for $\Mon_1(C,D)$ in place of $\Hom_1(C,D)$.

        If $\sigma^i(\sqrt{D})^r y_{i+1} \leq^* C$ for all shifts $\sigma^i$ of $\sqrt{D}$, as certainly happens when $1 < R/r$, then $\sigma^i(D)y_{i+1} \leq^* C$ for all shifts $\sigma^i$ of $D$,  and so by \Cref{MonRedEasyWay} the maximal vertex $\Phi_i^M$ of $\Mon_1(C,D;i)$ has an up edge to $\Mon_1(C,D;i+1)$, for each $i$.  
      Thus $\Mon_1(C,D)$  is connected and cyclic. By \Cref{uptomono2} $\Hom_1(C,D)$ is also connected and cyclic,  so part (2) of the theorem is proved. 
        
        For part (3) of the theorem, consider first the exceptional case that  $D$ is symmetric and $C$ is a directed cycle. If $m = n$ then all wind one maps are monotone, so           $\Hom_1(C,D;i) = \Mon_1(C,D;i)$ for all $i$,  and this graph contains exactly one map, $\phi_i$. By \Cref{lem:refine} $\Mon_1(C,D)$ is a backwards   directed cycle if $C$ is forwards directed, and is a forwards directed cycle if $C$ is backwards directed; either way, it has a single cyclic component, as needed.    When $m > n$ the exceptional case falls into  case (2) of the theorem, and so the exceptional case is proved.  

        We may therefore assume that $C$ and $D$ are not in the exceptional case.  By the assumptions of part (3) there is at least one maximal interval $I  = \{a,a+1, \dots, b\}\subset [n]$ such that $b$ is the only element $i$ in $I$ that does not satisfy $\sigma^i(\sqrt{D})^r y_{i+1} \leq^* C$.  
      We have by \Cref{MonRedEasyWay} and \Cref{uptomono2} that $X_I:= \cup_{\alpha \in I} \Mon^+_1(C,D;i)$ is connected for any such interval $I$
         and by \Cref{MonReduction} that $X_I$ has no up edges out of $X_I$.     As the same holds for the other such intervals $I'$, it also has no down edge, and so $X_I$ is a component of $\Hom_1(C,D)$.  
        Moreover, the maximal element $\Phi_j^M$ of $\Mon_1(C,D)$ has no up edges, as they would be in $\Mon^+_1(C,D;j+1)$,  and so $X_I$ is not cyclic.  
     \end{proof}

\section{Algorithms}\label{sec:Algorithms}

      In \cite{BLS21}, a polynomial time algorithm was given to solve the problem $\Recon(D)$ for a reflexive digraph $D$.   
      There was no effort in that paper to optimise this algorithm, but when applied to a cycle instance $(C, \phi, \psi)$ of size $|V(C)| = m$, it would essentially 
      move vertices greedily up towards their image under $\psi'$.  If this was not possible, it would try to get to $\psi$ by moving vertices greedily down. 
      The most a single vertex might have to move was a distance of about $2m$, and moving several vertices up by one, one had to check again what vertices 
      would be able to move by computing a graph $\Aux(\phi)$, much like the graph $\Neq(\phi,\phi')$ for some imagined up-neighbour $\phi'$ of $\phi$.
      Quantifying the running time of this algorithm, one could easily upper bound the running time as $O(m^3)$ but the lower bound would not beat  $\Omega(m^2)$.
      The algorithm would take space $O(m)$, having to keep track of a map of $C$ to $D$ at all times.  

      We show that we can solve the problem for cyclic instances $(C, \phi, \psi)$  of size $m$ in time $O(m)$ and space $O(\log(m))$. 
      Recall that $D$ is not part of the instance, so it can be assumed that $\sqrt{D}$ is given. While the size $s$ of $\sqrt{D}$ and 
      the size $n$ of $D$ are constant for the reconfiguration problem, we first consider the running time of the various algorithms in $m$ and in $n$ or $s$.

    \begin{fact}
        Given a digraph cycle $D$ of length $n$, we can find $\sqrt{D}$  and the value $r$ such that $D = \sqrt{D}^r$ in time $O(n^{3/2})$. 
    \end{fact}
    \begin{proof}
       As observed above, $r$ is $n/i$ where $i$ is the smallest positive integer for which $\sigma^i(D) = D$.  
       As $i$ must divide $n$, it is enough to compute $\sigma^i(D)$ for $i \leq \sqrt{n}$ and compare it to $D$, which 
       takes time $O(n)$ for each of at most $\sqrt{n}$ values of $i$.  
    \end{proof}

     With a couple variations of this simple algorithm, we get the following.    
     
    \begin{lemma}\label{lem:Complexity}
    Let $\sqrt{D}$ be a primitive cycle of fixed length $s$.  For a  given digraph cycle $C$ of length $m$, we have the following.
      \begin{enumerate} 
           \item  We can find the largest integer $R$ such that 
      $\sqrt{D}^R \leq^* C$ in time $O(m)$ and in space $O(\log(m))$.  
           \item We can find the largest integer $R$ such that $\sigma^i(\sqrt{D})^R \leq^* C$ for some $i \in \{0, \dots, s-1\}$ in time $O(sm)$ and space $O(\log(m))$.
           \item For given $r$, we can find the set $\Gamma$ of $i$ such that $\sigma^i(\sqrt{D})^r y_{i+1} \leq^* C$ in time $O(sm)$ and space $O(s\log(s) + \log(m))$.  
      \end{enumerate}
      \end{lemma}
    \begin{proof}
        
       Let $d$, initialised to $d  = 1$, be a pointer pointing at the index of an edge of $\sqrt{D} = y_{1}y_{2} \dots y_p$ that we are looking for, and let $c$, initialised to $c = 0$, be a counter of the number of edges of $\sqrt{D}$ we have found.
        For $\alpha = 1, \dots, m$, if we have  $y_{d} \leq^* x_\alpha$,  then increment $d$ modulo $p$ to the next edge of $P$, and increment $c$ by one.
        From the resulting $c$, we get $R  = \ceil{c/p}$. This takes time $O(m)$, and the integer $c$ is bounded by $m$, so can be held in space logarithmic in $m$.
        This  gives  part (1) of the lemma. 

        For part (2) of the lemma, we run  this algorithm on $\sigma^i(\sqrt{D})$ instead of $\sqrt{D}$ for each $i = 1, \dots p$. 
        Keeping only the maximum value $R_{\rm max}$ of the values $R$ returned for each $i$, before we increment $i$ we compare $R$ with $R_{\rm max}$ and 
        let $R_{\rm max}$ be the greater value.  This takes time $O(pm) = O(m)$ and space $O(\log(m))$. 

        For part (3) we run the algorithm on $\sigma^i(\sqrt{D})$ but instead of computing $R$ for each $i \in \{0, \dots, s-1\}$ we put $i$ in the set $\Gamma$
        if $c$ reaches $rs+1$ This takes time $O(sm)$ and as each index in $\Gamma$ takes space at most $\log(s)$, it takes space  $O(s\log(s) + \log(m))$.

    \end{proof}

    \begin{proposition}\label{Complexity}
      Given a digraph cycle $D = \sqrt{D}^r$ where $\sqrt{D}$ is a primitive non-contractible cycle, 
      we can solve an instance $(C,\phi,\psi)$ of $\Recon(D)$ of size $m$ in time $O(m)$ and space $O(\log m)$.  
    \end{proposition}     
    \begin{proof}
       First we determine the winds of the maps $\phi$ and $\psi$; this clearly takes time $O(m)$ and space $O(\log(m))$.
       If the winds are the same value $w$ for both maps, then we find the set $\Gamma$ of values of $i$ such that $\sigma^i(\sqrt{D}^{rw}) \leq^* C$;
       by part (3) of \Cref{lem:Complexity}, this takes time $O(m)$ and space $O(\log(m))$.  To decide if $\phi$ reconfigures to $\phi'$ we check 
       if all values of $i$ between $\phi(c_0)$ and $\psi(c_0)$ (in either direction) are in $\Gamma$.  This takes constant time and space.
    \end{proof}

\section{Concluding Remarks}

 We showed that the version of $\Recon(D)$ for a reflexive digraph cycle $D$, where we consider only cyclic instances, is solvable in log-space. 
 
 A natural next question is of whether the problem $\Recon(D)$ can be solved in log-space for more general instances.  From \cite{BLS21} it is enough to check the conditions for a single cycle instance on every cycle in a basis of the cycle space.
 However, the bookkeeping required to check every cycle of a cycle space seems to require more than log-space.

Another natural question is about more general digraph cycles.   From results about the reconfiguration problem $\Recon(H)$ for more general digraph targets $H$ in \cite{LMS} one gets that $\Recon(D)$ is also polynomial time solvable when $D$ is an irreflexive cycle.  We expect the techniques of the present paper could be used to show that this problem is also in log-space for cyclic instances; however, a characterisation seems messier.

 Finally, we note that in the recent paper \cite{FIKNS}, the authors consider a similar problem for irreflexive graphs, but replacing the graph $\Hom(G,D)$ with a simplicial complex;  they show that the homotopy type of each component is contractible or a wedge of spheres. Along these lines, for our setting, one might consider the homotopy type of the transitive tournament complex of $\Hom(C,D)$-- the complex whose simplices are transitive tournaments.  We expect that in \cref{thm:wind_components},
 our non-cyclic components yield contractible simplicial complexes, and that our cyclic components yield simplicial complexes that are homotopic to $S^1$.

\section*{Acknowledgements}

  Both authors are indebted to Beno\^{i}t Larose for his advice and support throughout the project, and for introducing us to each other and to the problem.


\begin{thebibliography}{5}

\providecommand{\bysame}{\leavevmode\hboxto3em{\hrulefill}\thinspace}
\newcommand{\doi}[1]{\href{http://dx.doi.org/#1}{\small\nolinkurl{DOI: #1}}}
\renewcommand{\url}[1]{\href{https://arxiv.org/abs/#1}{\small\nolinkurl{arXiv: #1}}}

\providecommand{\natexlab}[1]{#1}


\bibitem[Brewster et~al.(2018)Brewster, Lee, and Siggers]{BLS18}
R.~Brewster, J.~Lee, and M.~Siggers.
\newblock Recolouring reflexive digraphs.
\newblock \emph{Discrete Math.}, 341\penalty0 (6):\penalty0 1708--1721, 2018.
\newblock \doi{10.1016/j.disc.2018.03.006}.

\bibitem[Brewster et~al.(2021)Brewster, Lee, and Siggers]{BLS21}
R.~Brewster, J.~Lee, and M.~Siggers.
\newblock Reconfiguration of homomorphisms to reflexive digraph cycles.
\newblock \emph{Discrete Math.}, 344\penalty0 (8):\penalty0 112441, 2021.
\newblock \doi{10.1016/j.disc.2021.112441}.


\bibitem[Fujii et~al.(2024)Fujii, Iwamasa, Kimura, Nozaki, and Suzuki]{FIKNS}
Soichiro Fujii, Yuni Iwamasa, Kei Kimura, Yuta Nozaki, and Akira Suzuki.
\newblock {Homotopy types of Hom complexes of graph homomorphisms whose
  codomains are cycles}.
\newblock \emph{arXiv}, August 2024.
\newblock \doi{10.48550/arXiv.2408.04802}.

\bibitem[Larivi{\ifmmode\grave{e}\else\`{e}\fi}re
  et~al.(2024)Larivi{\ifmmode\grave{e}\else\`{e}\fi}re, Larose, and
  Pullas]{LLP24}
Isabelle Larivi{\ifmmode\grave{e}\else\`{e}\fi}re,
  Beno{\ifmmode\hat{\imath}\else\^{\i}\fi}t Larose, and David
  E.~Pazmi{\ifmmode\tilde{n}\else\~{n}\fi}o Pullas.
\newblock {Surjective polymorphisms of directed reflexive cycles}.
\newblock \emph{Algebra Universalis}, 85\penalty0 (1):\penalty0 1--28, February
  2024.
\newblock ISSN 1420-8911.
\newblock \doi{10.1007/s00012-023-00834-4}.

\bibitem[L\'{e}v\^{e}que et~al.(2023)L\'{e}v\^{e}que, M\"{u}hlenthaler, and
  Suzan]{LMS}
Benjamin L\'{e}v\^{e}que, Moritz M\"{u}hlenthaler, and Thomas Suzan.
\newblock {Reconfiguration of Digraph Homomorphisms}.
\newblock In Petra Berenbrink, Patricia Bouyer, Anuj Dawar, and
  Mamadou~Moustapha Kant\'{e}, editors, \emph{40th International Symposium on
  Theoretical Aspects of Computer Science (STACS 2023)}, volume 254 of
  \emph{Leibniz International Proceedings in Informatics (LIPIcs)}, pages
  43:1--43:21, Dagstuhl, Germany, 2023. Schloss Dagstuhl -- Leibniz-Zentrum
  f{\"u}r Informatik.
\newblock ISBN 978-3-95977-266-2.
\newblock \doi{10.4230/LIPIcs.STACS.2023.43}.

\bibitem[Mar{\ifmmode\acute{o}\else\'{o}\fi}ti and
  Z{\ifmmode\acute{a}\else\'{a}\fi}dori(2012)]{MZ12}
M.~Mar{\ifmmode\acute{o}\else\'{o}\fi}ti and
  L.~Z{\ifmmode\acute{a}\else\'{a}\fi}dori.
\newblock {Reflexive digraphs with near unanimity polymorphisms}.
\newblock \emph{Discrete Math.}, 312\penalty0 (15):\penalty0 2316--2328, 2012.
\newblock \doi{10.1016/j.disc.2012.03.040}.

\end{thebibliography}
\end{document}